\documentclass[a4paper,10pt]{amsart}
\usepackage{amsmath,amsfonts,amssymb, amsthm, xypic, color}

\usepackage{epsfig}
\usepackage{verbatim}
\def\ZZ{\mathbb{Z}}

\def\UU{\boldsymbol{\mathcal{E}}}

\def\T{\mathbb{T}}

\def\can{\mathbf{can}}

\def\tE{\widetilde{\UU}}
\def\utE{\underline{\tE}}
\def\u{\mathfrak{u}}

\def \ot{\otimes}
\def\<{\langle}
\def\>{\rangle}

\newtheorem{thm}{\bf{Theorem}}[section]
\newtheorem{lemma}[thm]{Lemma}

\theoremstyle{remark}

\theoremstyle{definition}

\numberwithin{equation}{section}

\title[Addendum to Drinfeld Realization of the Elliptic Hall algebra]{Addendum to Olivier Schiffmann, ``Drinfeld Realization of the Elliptic Hall Algebra''}
\author{Dragos Fratila}
\address{Universit\'e Paris Denis-Diderot - Paris 7, Institut de Mathematiques de Jussieu, UMR 7586 du CNRS, Batiment Chevaleret, 75205 Paris Cedex 13, France}
\email{fratila@math.jussieu.fr}

\begin{document}
\maketitle

\begin{abstract}In~\cite{S} O. Schiffmann gave a presentation of the Drinfel'd double of the elliptic Hall algebra which is similar in spirit to Drinfel'd's new realization of quantum affine algebras. Using this result together with a part of his proof we can provide such a description for the elliptic Hall algebra.
\end{abstract}

\vspace{0.1in}

We will use freely all the notations and the results of~\cite{S}.

% Consider the formal series
% \[
%  \widetilde{\T}_1(z) = \sum_{d\in\ZZ}\u_{1,d}z^d
% \]
% \[
%  \widetilde{\T}_0^+(z) = 1+\sum_{d\ge 1}\Theta_{0,d}z^d
% \]

Let $\utE^+$ be the algebra generated by the Fourier coefficients of the series $\mathbb{T}_{1}(z)$ and $\mathbb{T}_0^+(z)$ subject only to the relevant positive relations (4.1), (4.2), (4.3), (4.5) in~\cite{S}. To avoid any confusion with the generators of $\tE$ we denote the generators of $\utE^+$ by $\u_{1,d},d\in\ZZ$ and $\Theta_{0,d},d\ge 1$. 

We denote by $\tE^\pm$ the \textbf{subalgebra} of $\tE$ generated by the positive (resp. negative) generators. Similarly for $\UU^\pm$.
Our goal is to prove that $\UU^+$ is isomorphic to $\utE^+$. The strategy is to go through their Drinfel'd doubles. But first we need to define a coalgebra structure on $\utE^+$. 

\begin{lemma}\label{bialgebra} The map $\Delta:\utE^+\to\utE^+ \widehat\ot\, \utE^+$ given on generators by
\[
 \Delta(\T_0^+(z)) = \T_0^+(z)\ot\T_0^+(z)
\]
\[
 \Delta(\T_1(z)) = \T_1(z)\ot 1+\T_0^+(z)\ot \T_1(z)
\]
is a well defined algebra map and makes $\utE^+$ into a (topological) bialgebra.
\end{lemma}
\begin{proof}
We need to check that the map $\Delta$ respects all the relations between the generators of $\utE^+$. 
The relations (4.1), (4.2), (4.3) are an easy routine check. We are left to check the cubic relation (4.5). Using \cite{S} Lemma~4.1 we only need to check the following relation:
\[
 [[\u_{1,-1},\u_{1,1}],\u_{1,0}]=0.
\]
Applying $\Delta$ we obtain:
\begin{equation}\label{E:delta cubic}
 [[\u_{1,-1},\u_{1,1}],\u_{1,0}]\ot 1 + E + \sum_{m,n,l\ge 0} \Theta_{0,m}\Theta_{0,n}\Theta_{0,l}\ot [[\u_{1,-1-m},\u_{1,1-n}],\u_{1,-l}]
\end{equation}
where $E \in \utE^+[1]\widehat\ot \utE^+[2]+\utE^+[2]\widehat\ot \utE^+[1]$.

The first term is $0$ since it's exactly the cubic relation. We want to prove that $E$ and the third term are also 0. Let us begin with $E$.

We will need to use the following easy lemma whose proof is omitted:
\begin{lemma}
 Let $A,B$ be two algebras over a field. Suppose we have a morphism of algebras $f:A\to B$. Then $\ker(f\ot f) = A\ot \ker(f)+\ker(f)\ot A$.
\end{lemma}

The arguments of~\cite{S} Section 5.3 show that $\utE^+[\le 2]$ and $\UU^+[\le 2]$ are isomorphic (through the canonical morphism).
We apply the above lemma to this morphism $\can:\utE^+\to \UU^+$ and we get in particular that
\[
 \utE^+[\le 2]\ot \utE^+[\le 2]\to\UU^+[\le 2]\ot\UU^+[\le2]
\]
is still an isomorphism.

Using the fact that the map $\can$ commutes with the coproduct we get that $\can\ot\can(E)=0$. By the above isomorphism we deduce that $E=0$.\footnote{It looks like we cheated here because $E$ lives only in a completion of the tensor product. However, each graded piece of $E$ (remember that $\utE^+$ is $\ZZ^2$ graded) lives in an ordinary tensor product and hence we can apply the lemma.}

Let us now deal with the cubic term. 
For any integers $m,n,l\in\ZZ$ we put 
\[
R(m,n,l) = \sum_{(m,n,l)} [[\u_{1,-1+m},\u_{1,1+n}],\u_{1,l}]
\] where the sum is over all the six permutations of the triplet $(m,n,l)$.
So in order to prove that the third term of the relation~(\ref{E:delta cubic}) vanishes it is enough to prove that $R(m,n,l)=0$ for any $m,n,l\in\ZZ$.

Observe first that $R(l,l,l)=0$ for any $l\in\ZZ$ since it is the cubic relation (4.6) from~\cite{S}. By symmetry we can suppose that $l\le m,n$. Applying the adjoint action of $\u_{0,k-l}$ to the relation $R(l,l,l) = 0$  we get that $R(k,l,l) = 0$ for any $k\ge l$. So in particular $R(m,l,l) = 0$. Now applying the adjoint action of $\u_{0,n-l}$ to $R(m,l,l)=0$ we obtain $R(m,n,l)=0$ which is exactly what we wanted.
\end{proof}

In~\cite{S} it is proved that $\tE^+$ is isomorphic to $\UU^+$. It follows that there is a natural surjective morphism $\pi:\utE^+\to\tE^+\simeq \UU^+$ and therefore a natural surjective morphism on the Drinfel'd doubles:
\[
 D\utE^+\to D\UU^+\simeq \UU \simeq \tE
\]

If the natural map $\tE\to D\utE^+$ is well defined then since the composition 
\[
\tE\to D\utE^+\to \tE                                                                                                                                                                                                                                                                                                                                                               \]
is the identity (because all the morphisms are the obvious ones) we obtain that 
\[
 \tE^+\simeq\utE^+
\]
which is what we wanted.

To prove that the natural morphism $\tE\to D\utE^+$ is well defined we need to check that the relations (4.1)-(4.5) are satisfied in $D\utE^+$. It is clear that (4.1), (4.3), (4.5) and (4.2) ($\epsilon_1=\epsilon_2$) are satisfied since they involve only the positive (resp. negative) part at once. We need to deal with (4.2) ($\epsilon_1=-\epsilon_2$) and (4.4). We claim that they are implied by  Drinfel'd's relations in the double. This is an easy verification.

Putting all together we have:
\begin{thm}
 The elliptic Hall algebra $\UU^+$ is isomorphic to the algebra generated by the Fourier coefficients of $\T_1(z)$ and $\T_0^+(z)$ subject to the relations:
\[
 \T_0^+(z)\T_0^+(w) = \T_0^+(w)\T_0^+(z)
\]
\[
 \chi_1(z,w)\T_0^+(z)\T_1(w) = \chi_{-1}(z,w)\T_1(w)\T_0^+(z)
\]
\[
 \chi_1(z,w)\T_1(z)\T_1(w) = \chi_{-1}(z,w)\T_1(w)\T_1(z)
\]
\[
 \mathrm{Res}_{z,y,w}[(zyw)^m(z+w)(y^2-zw)\T_1(z)\T_1(y)\T_1(w)] = 0,\,\forall m\in\ZZ
\]
\end{thm}

\begin{center}\textsc{Acknoledgements}
\end{center}
I am indebted to Olivier Schiffmann for suggesting the solution to the cubic term issue. I would also like to thank Alexandre Bouayad for numerous discussions on the Drinfel'd double.


\begin{thebibliography}{9}
 \bibitem{S} O. Schiffmann - \emph{Drinfeld realization of Elliptic Hall Algebra}, to appear in Journal of Algebraic Combinatorics, (2011)
\end{thebibliography}
\end{document}